\documentclass[a4paper]{article}
\usepackage{amssymb,amsbsy,amsmath,amsfonts,amssymb,amscd}
\usepackage{latexsym}

\newtheorem{proposition}{Proposition}
\newtheorem{lemma}[proposition]{Lemma}

\newtheorem{theorem}{Theorem}
\newenvironment {proof}%
{ \noindent {\em Proof. }}%
{\hspace*{\fill}$\Box$\par \medskip }

\def \Mid{\quad \vrule \quad}
\def\BZ{{\hbox{$\mathbb{Z}$}}}

\def\BF{{\hbox{$\mathbb{F}$}}}
\def\BN{{\hbox{$\mathbb{N}$}}}
\def\BR{{\hbox{$\mathbb{R}$}}}
\def\BC{{\hbox{$\mathbb{C}$}}}
\def\BQ{{\hbox{$\mathbb{Q}$}}}
\def\BA{{\hbox{$\mathbb{A}$}}}

\begin{document}

\title{Natural boundaries for Euler products of 
Igusa zeta functions of elliptic
curves.}
\author{Marcus du Sautoy\thanks{%
Mathematical Institute, Andrew Wiles Building, Woodstock Road, Oxford OX2 6GG, UK. dusautoy@maths.ox.ac.uk}}
\maketitle

\begin{abstract} 

We study the analytic behaviour of adelic versions of Igusa integrals given
by integer polynomials defining elliptic curves. 

\medskip
\hrule 

\hspace{1cm} \\

2010 Mathematics Subject Classification: 11M41 
\end{abstract}

\tableofcontents

\section{Introduction}

Let $f(x,y)\in \BZ[x,y]$ be a nonsingular 
integer polynomial defining an elliptic curve
$E$. The Igusa zeta function of $f$ at the prime $p$ is defined as:
\begin{equation}\label{Ig1}
I_{E}(s,p)=\int_{\BZ_{p}^{2}}|f(x,y)|_p^{s}\, d\mu
\end{equation}
where $|\ |_p$ stands for the p-adic absolute value and $d\mu$ for the
Haar-measure on $\BQ_p^2$ normalized such that $\BZ_p^2\subset 
\BQ_p^2$ has measure equal to $1$. 

Let $S(E)$ be the set of those primes $p$ 
for which the reduction of $f(x,y)$ modulo $p$ becomes singular. Note 
that $S(E)$ is a finite set of primes. For the primes $p\notin S(E)$ the
integral (\ref{Ig1}) can be computed as:
\begin{equation}\label{Ig2}
I_{E}(s,p)=(1-p^{-2}C_{p})\left( 1+\lambda _{p}\frac{p^{-(s+1)}}{%
(1-p^{-(s+1)})}\right)
\end{equation}
(see Proposition \ref{comp1}) where the numbers $C_p$ and $\lambda_p$ 
are given by:
\begin{equation}\label{Ig3}
C_{p}=\mathrm{card}\left\{ (x,y)\in\BZ/p\BZ 
\Mid f(x,y)=0\ {\rm mod}\ p\right\}, \qquad
\end{equation}

\begin{equation}\label{Ig4}
\lambda _{p}=\frac{(p-1)C_{p}}{(p^{2}-C_{p})}.
\end{equation}
Following Ono \cite{Ono} we form the Euler-product
\begin{equation}\label{Ig5}
I_{E}(s)=\prod_{p\notin S(E)}I_{E}(s,p)\, (1-p^{-2}C_{p})^{-1}
\end{equation}
which have their constant coefficients equal to $1.$

We call $I_E(s)$ the global Igusa-zeta function of the elliptic curve given by
$f$. Given this notation we show:

\begin{theorem}\label{merotheo}
\begin{itemize}

\item[(i)] The Euler product $I_E(s)$ converges for $s\in\BC$ with $\Re(s) >0$,

\item[(ii)] $I_E(s)$ has a meromorphic continuation to the region 
 $\Re(s) >-3/2$,

\item[(iii)]  (assuming GRH) the line $\Re (s)=-3/2$ is a
natural boundary for $I_{E}(s)$ beyond which no further meromorphic
continuation is possible.

\end{itemize}

\end{theorem}

In the case that $E$ has no complex multiplication, the result depends on the recent proof of the Sato-Tate conjectures together with the associated proof about the meromorphic continuation of symmetric power $L$-functions (see \cite{CHT}, \cite{HS-BT} and \cite{Taylor}).

Interest in proving such a theorem arose from considerations of meromorphic continuation and natural boundaries for zeta functions of groups and rings. \cite{duS-Natural boundary}.

Grunewald, Segal and Smith introduced the notion of the zeta function of a
group $G$ in \cite{GSS}: 
\[
\zeta _{G}^{\leq }(s)=\sum_{H\leq G}|G:H|^{-s}=\sum_{n=1}^{\infty
}a_{n}^{\leq }(G)n^{-s} 
\]
where $a_{n}^{\leq }(G)$ denotes the number of subgroups of index $n$ in $G.$

They proved that for finitely generated, torsion-free nilpotent
groups the global zeta function can be written as an Euler product of local
factors which are rational functions in $p^{-s}:$%
\begin{eqnarray*}
\zeta _{G}^{\leq }(s) &=&\prod_{p\text{ prime}}\zeta _{G,p}^{\leq }(s) \\
&=&\prod_{p\text{ prime}}Z_{p}^{\leq }(p,p^{-s})
\end{eqnarray*}
where for each prime $p,$ $\zeta _{G,p}^{\leq }(s)=\sum_{n=0}^{\infty
}a_{p^{n}}^{\leq }(G)p^{-ns}$ and $Z_{p}^{\leq }(X,Y)\in \Bbb{Q}(X,Y).$

Similar definitions and results were also obtained for the zeta function $%
\zeta _{G}^{\triangleleft }(s)$ counting normal subgroups. Analogous results also hold for zeta functions counting ideals or subalgebras in finite dimensional Lie algebras.

In \cite{duSG-analytic} the author and Grunewald were able to prove that these zeta functions have some meromorphic continuation beyond their radius of convergence. 

\begin{theorem}
Let $G$ be a finitely generated nilpotent
group. Then the abscissa of convergence $\alpha(G)$ of $\zeta_G(s)$ is a
rational number and $\zeta_G(s)$ has a meromorphic continuation to
the right half-plane ${\rm Re}(s)>\alpha(G)-\delta$ for suitable $\delta >0$.
\end{theorem}

The challenge was then to determine how far one could meromorphically continue such functions or whether at some point you encountered a natural boundary beyond which no continuation was possible. Many of the early examples of calculations of these functions revealed a uniform behaviour of the rational functions as one varied the prime $p$. In particular the zeta function of a nilpotent group $G$ is called uniform if there exists $W(X,Y)\in\BQ(X,Y)$ such that for almost all primes $p$
\[
\zeta _{G,p}^{\leq }(s)=W(p,p^{-s}).
\]

In Chapter 5 of \cite{duS-Natural boundary} given a two variable polynomial $W(X,Y),$ we provided a criterion for the
Euler product $\prod_{p\text{ prime}}W(p,p^{-s})$ to have a natural boundary
beyond which one cannot continue the function meromorphically. This
partially answered a conjecture characterizing those polynomials which do
admit meromorphic continuation to the whole complex plane, generalizing a
result of Estermann's for one variable polynomials \cite{Estermann}.

However in \cite{duSElliptic1} and \cite{duSElliptic2} we proved that not all finitely generated nilpotent groups have uniform zeta functions. In particular we proved the following:  

\begin{theorem}
Let $L$ be the class two nilpotent Lie algebra over $\Bbb{Z}$ of dimension 9
as a free $\Bbb{Z}$-module given by the following presentation: 
\[
L=\left\langle 
\begin{array}{c}
x_{1},x_{2},x_{3},x_{4},x_{5},x_{6},y_{1},y_{2},y_{3}:(x_{1},x_{4})=y_{3},(x_{1},x_{5})=y_{1},(x_{1},x_{6})=y_{2}
\\ 
(x_{2},x_{4})=y_{1},(x_{2},x_{5})=y_{3},(x_{3},x_{4})=y_{2},(x_{3},x_{6})=y_{1}
\end{array}
\right\rangle 
\]
where all other commutators are defined to be 0. Let $E$ be the elliptic
curve $Y^{2}=X^{3}-X.$ Then there exist two non zero rational functions $%
P_{1}(X,Y)$ and $P_{2}(X,Y)\in \Bbb{Q}(X,Y)$ such that for almost all primes 
$p$: 
\[
\zeta _{L_{p}}^{\triangleleft }(s)=P_{1}(p,p^{-s})+|E(\Bbb{F}%
_{p})|P_{2}(p,p^{-s}).
\]
\end{theorem}

Interest therefore turned to whether the global zeta function of this example whose local factors depend on the number of points on an elliptic curve might have natural boundaries beyond which no meromorphic continuation is possible.

Given that the main tool in proving these results is the use of $p$-adic integrals similar to those defining the Igusa zeta function it seemed apposite to first consider the global nature of a function defined as an Euler product of local Igusa zeta functions. Such global functions have received little attention beyond an initial paper by Ono \cite{Ono} where some meromorphic continuation of this global zeta function is demonstrated. This paper can be considered as a continuation of that early paper of Ono's, establishing that in general these global Igusa functions have natural boundaries.

I should like to dedicate this paper to the memory of Fritz Grunewald with whom I discussed many of the ideas at the heart of this paper and beyond.

\section{Some polynomial identities}\label{polyid}

This section contains several polynomial identities which will be important
for the meromorphic continuation of the global Igusa zeta functions.

We write $\BN$ for the natural numbers and $\BN_0$ with zero included. 
Let ${\cal L}$ be the set of $3$-tuples $(r,n,m)\in \BN_0^{2}\times \BN$ 
with the property that $\left( r+2n\right) /2m=1/2$ or equivalently
with $r+2n=m$. 
For each $M\in\BN$ let ${\cal L}_{M}$ denote the finite set
of $(r,n,m)\in {\cal L}$ for which $m<M.$ 
For $(r,n,m)\in {\cal L}$ we define its weight to
be: 
\begin{equation}\label{wei}
\Theta(r,n,m):=\frac{r+2n+2}{2m}=\frac{1}{2}+\frac{1}{m}.
\end{equation}

We consider here the free commutative $\BZ$-algebra $R$ 
generated by $u,\, v,\, X,\, Y$ with the single relation $uv=X$, that is  
\begin{equation}\label{ring}
R:=\BZ\langle\,  u,\, v,\, X, \, Y\Mid uv=X\, \rangle.
\end{equation}
Every monomial in $u,\, v,\, X,\, Y$ in $R$ is equal to 
$u^{r_1}v^{r_2}X^{n}Y^{m}$ with $r_1,\, r_2,\, n,\, m\in \BN_0$ and with at
least one of the $r_1,\, r_2$ equal to zero. 
Every $Q\in R$ which is symmetric in $u,\, v$ can be written as 
\begin{equation}\label{deco0}
Q=\sum_{(r,n,m)\in\BN_0^3} a_{Q,(r,n,m)}(u^r+v^r)X^nY^m
\end{equation}
with uniquely determined coefficients 
$a_{Q,(r,n,m)}\in\BZ$, $a_{Q,(r,n,m)}=0$ for almost all $(r,n,m)\in\BN_0^3$.

For $(r,n,m)\in {\cal L}$ and $\varepsilon\in\{\pm 1\}$ 
with $r\ne 0$ we define:
\begin{equation}\label{poly}
P^{[\varepsilon]}_{(r,n,m)}:=1-\varepsilon(u^r+v^r)X^nY^m+X^{r+2n}Y^{2m}
\end{equation}
For $(r,n,m)\in {\cal L}$ with $r= 0$ and $\varepsilon\in\{\pm 1\}$ 
we define:
\begin{equation}\label{poly2}
P^{[\varepsilon]}_{(0,n,m)}:=1-\varepsilon X^{n}Y^{m}
\end{equation}

We establish now the following cyclotomic expansion of $1-uY-vY.$

\begin{lemma}\label{deco1}
For $w=(r,n,m)\in {\cal L}$ there exist integers $c(w)\in \BN_0$ and 
$\varepsilon{(w)}\in \{-1,1\}$ with the property that for 
$M\in \BN$
$$1-uY-vY=\left( \prod_{w\in {\cal L}_M} 
(P^{[\varepsilon(w)]}_{w})^{c(w)}\right) + W_M(u,v,X,Y)$$
where $W_M$ is symmetric with respect to $u,\, v$ and $(r,n,m)\in 
{\cal L}\setminus {\cal L}_M$ for every $(r,n,m)\in \BN_0^3$ for which the
coefficient $a_{W_M,(r,n,m)}$ from (\ref{deco0}) is non-zero.
\end{lemma}

Before we give the proof of the lemma we shall give several examples

\bigskip

\centerline{M=2}
$$1-uY-vY=P^{[1]}_{(1,0,1)}-XY^2,$$
\centerline{M=3}
$$1-uY-vY=P^{[1]}_{(1,0,1)}P^{[1]}_{(0,1,2)}-(u+v)XY^3 + X^2Y^4,$$
\centerline{M=4}
$$1-uY-vY=P^{[1]}_{(1,0,1)}P^{[1]}_{(0,1,2)}P^{[1]}_{(1,1,3)}
+(u^2+v^2)X^2Y^6 $$
$$- (u^2+v^2)XY^4 - (u+v)X^4Y^9 +X^5Y^{10} + X^3Y^6 - X^2Y^4.$$

\bigskip

We note the following identity: for $n\geq m$%
\begin{equation}
(u^{n}+v^n)(u^{m}+v^m)=(u^{n+m}+v^{n+m})+(u^{n-m}+v^{n-m})X^{m}.  \label{U identity}
\end{equation}

\textbf{Proof of Lemma \ref{deco1}} This follows by induction on $M.$ We start the induction by
writing:
\[
1-uY-vY=P^{[1]}_{(1,0,1)}-XY^2. 
\]
To deal with the inductive step we need to show how to remove a term of the
form $a_{r,n,M}(u^{r}+v^r)X^{n}Y^{m}$ in $W_{M}(u,v,X,Y)$ at the expense of
introducing things of the form $a_{r,n,m}(u^{r}+v^r)X^{n}Y^{m}$ with $m>M.$ If $%
r>0,$ then we introduce a new term
\[
\left( 1+(-1)^{\varepsilon _{r,n,M}}(u^{r}+v^r)X^{n}Y^{M}+X^{r+2n}Y^{2M}\right)
^{c_{r,n,M}} 
\]
where $a_{r,n,M}=(-1)^{\varepsilon _{r,n,M}}c_{r,n,M}$ and $c_{r,n,M}>0.$
The error terms that we must introduce into $W_{M+1}(u,v,X,Y)$ are then of the
form
\[
c(u^{r_{1}}+v^{r_{1}})\cdots (u^{r_{l}}+v^{r_{l}})X^{i}Y^{m}
\]
where $m>M.$ We can rewrite this using the identity (\ref{U identity}) as a
sum of terms of the form $c^{\prime }(u^{r}+v^r)X^{n}Y^{m}$ with the property that
$\left( r+2n\right) /2m=1/2.$ (Note that if $a/b=c/d=1/2$ then $%
(a+c)/(b+d)=1/2.)\square $
\bigskip

Using Lemma \ref{deco1} we define for $M\in\BN$

\begin{equation}\label{ku}
Q_M(u,v,X,Y):=\prod_{w\in {\cal L}_M} 
(P^{[\varepsilon(w)]}_{w})^{c(w)}\in R
\end{equation}

Suppose $S$ is a finite set of prime numbers and for every prime $p\notin S$
we are given a $\pi_p\in\BC$ with $\pi_p\overline{\pi_p}=p$. We consider then
the Euler products 
\begin{equation}\label{eule1}
Z^1_M(s):=\prod_{p\notin S} (1+W_M(\pi_p,\overline{\pi_p},p,p^{-s})
\end{equation}

\begin{equation}\label{eule2}
Z^2_M(s):=\prod_{p\notin S} \left(1+\frac{W_M(\pi_p,\overline{\pi_p},p,p^{-s})}
{Q_M(\pi_p,\overline{\pi_p},p,p^{-s})}\right)
\end{equation}

\begin{lemma}\label{eulel1}
Both the Euler products $Z^1_M(s)$ and $Z^2_M(s)$ converge absolutely for
$s\in \BC$ with $\Re(s)>1/2+1/M$
\end{lemma}

\bigskip

For $n\in\BN$ we define:
\begin{equation}\label{polys}
S^{[n]}:=\prod_{i=0}^n(1-u^{i}v^{n-i}Y)
\end{equation}

We have the recursion relations:
\begin{equation}\label{recpolysg}
S^{[2n]}=P^{[1]}_{(0,n,1)}\cdot \prod_{i=1}^nP^{[1]}_{2i,n-i,1}
\end{equation}

We have the recursion relations:
\begin{equation}\label{recpolysug}
S^{[2n-1]}=\prod_{i=1}^nP^{[1]}_{2i-1,n-i,1}
\end{equation}

\section{Some meromorphic functions\label{section: meromorphic functions}}

This section accumulates certain facts about the analytic properties of
of L-functions of Hecke-Groessencharacters and for symmetric power
L-functions

\subsection{L-functions of Hecke-grossencharacters}\label{gross}

Our basic reference for the 
terminology and theory of quasi-characters of Idele-class groups
is \cite{Lang}.

Let $K\subset\BC$  be an imaginary quadratic number field, 
$\BA_K^*$ its group of ideles and $J_K:=K^*\backslash \BA_k^*$ its group of
idele classes. We identify the set of finite places of $K$ with the set of
prime ideals in the ring of integers ${\cal O}$ of $K$. 
Given such a prime ideal $\wp$ we choose 
a generator $\pi_\wp$ of the maximal ideal of the completion ${\cal
  O}_\wp$. We define $[\wp]\in J_K$ to be the class of the idele which is $1$
at all places except for $\wp$ where it is  $\pi_\wp$.

Let
\begin{equation}
\chi: J_K\to \BC^* 
\end{equation}
be a (continuous) quasi-character and  $S(\chi)$ the set of non-archimedian
places where $\chi$ is ramified. The Hecke L-series of $\chi$ is defined as

\begin{equation}
L(\chi,s):=\prod_{\wp\notin S(\chi)} \, (1-\chi([\wp])N(\wp)^{-s})^{-1}
\end{equation}

\begin{proposition}\label{grossprop} 
Let $\chi$ be a quasi-character of the idele-class group with the property
that $|\chi([\wp])|={N(\wp)}^{\frac{1}{2}}$ for all $\wp\notin S(\chi)$.
Assume also that $r$ is a natural number.
The following hold: 
\begin{itemize}

\item[(i)] The L-function $L(\chi^r,s)$ converges for $\Re(s)>2r$.

\item[(ii)] The L-function $L(\chi^r,s)$ extends to a holomorphic function on
  all of $\BC$.

\item[(iii)] The zeroes of the extended L-function  $L(\chi^r,s)$ in 
$\{\, s\in\BC\Mid \Re(s)<0\,\}$ lie on the real axis.
\end{itemize}

\end{proposition} 

These functions will be sufficient to do meromorphic continuation of our global Igusa zeta function in the case that the elliptic curve has complex multiplication. Otherwise we need the following section.

\subsection{Symmetric power L-functions}\label{power}

We report here on the recent progress proving results about the meromorphic properties of the symmetric
power L-functions (see \cite{CHT}, \cite{HS-BT} and \cite{Taylor}). We also refer to \cite{Sha} for another good reference.  

We write $\BA$ for the ring of adeles of $\BQ$. Let further $S$ be a finite
set of places including the archimedian place. We shall identify the
non-archimedian places with the corresponding prime numbers. 
Let $\omega$ be a cuspidal
representation of ${\bf GL}(2,\BA)$. Assume that the associated local
representation $\omega_p$ is of class $1$ and also tempered 
for all $p\notin S$. For  
$p\notin S$ let
\begin{equation}\label{cpm}
t_p=\left( \begin {array}{cc}
                \alpha_p& 0 \\
                0 & \beta_p
        \end {array} \right) \in {\bf GL}(2,\BC)
\end{equation}
be the conjugacy class parametrizing $\omega_p$. We use here the normalisation
chosen in \cite{Sha}, Section 5, that is $|\alpha_p|=|\beta_p|=1$ and
$\alpha_p\beta_p=1$. 

Let $m$ be a natural number. For $p\notin S$ the 
local factor of the $m$-th symmetric power
L-function attached to $\omega$ is defined by:
\begin{equation}\label{symmi}
L(s,\omega_p,m):=\prod_{i=0}^m(1-\alpha_p^i\beta_p^{m-i}p^{-s})^{-1}
\end{equation}
The (partial) global $m$-th symmetric power
L-function is defined as the Euler product:

\begin{equation}\label{symmig}
L(s,\omega,m):=\prod_{p\notin S} L(s,\omega_p,m)
\end{equation}

The Dirichlet series $L(s,\omega,m)$ converges for $\Re (s)>1$. The results of \cite{CHT}, \cite{HS-BT} and \cite{Taylor} prove the following: 
\medskip

\begin{theorem}\label{CS1} $L(s,\omega,m)$ has a meromorphic continuation to all of
  $\BC$ for every $m\in\BN$.
  
\end{theorem}

Assume now that the automorphic represention $\omega$ is attached to a
holomorphic cusp form $g$ of weight $k$ for the congruence subgroup
$$\Gamma_0(N):=\left\{ \ \left( \begin {array}{cc}
                a & b \\
                c & d
        \end {array} \right) \in {\bf SL}(2,\BZ)\Mid c \equiv 0\ \hbox{\rm mod
                } \ N\ \right\}$$

$$\alpha_p=\frac{\pi_p}{p^{k-1/2}},\qquad \alpha_p=\frac{\bar\pi_p}{p^{k-1/2}}
$$

\begin{proposition}\label{Zmero}
Let $g$ be a holomorphic cusp form of weight 2 
for $\Gamma_0(N)$. Define for $r\in \BN$ and $\varepsilon\in{\pm 1}$:
\[
Z^{[\varepsilon]}_{r}(s)=\prod_{p \notin S(g)}Z^{[\varepsilon]}_{r,p}(s)
\]
where
\[
Z^{[\varepsilon]}_{r,p}(s)=1-\varepsilon\left( \pi_p ^{r}+\overline{\pi_p }^{r}\right) p^{-s}+p^{r-2s}.
\]
Then $Z^{[\varepsilon]}_{r}(s)$ is a meromorphic function on $\BC.$ 
\end{proposition}

For $\varepsilon =1$ follows from the recursive relations of the previous section.

For $\varepsilon =-1$ use the fact that 
\[ Z^{[-1]}_{r,p}(s)=Z^{[1]}_{2r,p}(2s)/Z^{[1]}_{r,p}(s)
\]

\section{Meromorphic continuation}\label{merocont}

Let $f(x,y)\in\BZ[x,y]$ 
denote a non-singular polynomial and let $S(f)$ be the set of primes for which
the reduction of $f$ in $\BZ/p\BZ$ becomes singular. 
As in the introduction we define the local Igusa integral at the prime $p$ by: 
\[
I_{f}(s,p)=\int_{\BZ_{p}^{2}}|f(x,y)|^{s}|dx||dy|.
\]

We have:

\begin{proposition}\label{comp1} Let $f(x,y)\in\BZ[x,y]$ 
be a non-singular polynomial and $p\notin S(f)$ then
\[
I_{f}(s,p)=(1-p^{-2}C_{p})\left( 1+\lambda _{p}\frac{p^{-(s+1)}}{%
(1-p^{-(s+1)})}\right)
\]
where
\[
\lambda _{p}=\frac{(p-1)C_{p}}{(p^{2}-C_{p})}.
\]
and 
$C_{p}=\mathrm{card}\left\{\, (x,y)\in\BZ /p\BZ 
\Mid f(x,y)=0 \ {\rm mod}\ p\, \right\} .$
\end{proposition}

\begin{proof} Let 
\[
E_0=\left\{ (x,y)\in\BZ_p^2\Mid f(x,y)\neq 0 \ {\rm mod}\ p\, \right\}
\]
\[
E_1=\left\{ (x,y)\in\BZ_p^2\Mid f(x,y)= 0 \ {\rm mod}\ p\, \right\}
\]
\[
E_n=\left\{ (x,y)\in\BZ_p^2\Mid f(x,y)= 0 \ {\rm mod}\ p^n\, \right\}
\]
Let 
\[
\tilde{E_n}=\left\{ (x,y)\in\left( \BZ_p/p^n\BZ_p\right)^2 \Mid f(x,y)= 0 \ {\rm mod}\ p^n\, \right\}
\]
and put $C_{p^n}=\mathrm{card}\tilde{E_n}$. Since $f(x,y)$ is non-singular $C_{p^n}=p^{n-1}C_p$.
\begin{eqnarray*}
I &=& \sum_{n=0}^\infty p^{-ns}\cdot\int_{E_n\setminus E_{n+1}} \\
&=& p^{-2}C_1+\sum_{n=1}^\infty p^{-ns}\cdot \left(C_{p^n}p^{-2n}-C_{p^{n+1}}p^{-2n-2}\right) \\
&=& p^{-2}C_1+\sum_{n=1}^\infty p^{-ns}\cdot C_pp^{(n-1)-2n}\left( 1-p^{-1}\right) \\
&=& \left( 1-p^{-2}C_p \right) + p^{-1}\left( 1-p^{-1}\right) \dfrac{p^{-(s+1)}}{(1-p^{-(s+1)}} \\
&=& (1-p^{-2}C_{p})\left( 1+\lambda _{p}\frac{p^{-(s+1)}}{%
(1-p^{-(s+1)})}\right)
\end{eqnarray*}
where
\[
\lambda _{p}=\frac{(p-1)C_{p}}{(p^{2}-C_{p})}.
\]

\end{proof}

\bigskip

The global Igusa zeta function of $f$ is then
\[
I_{f}(s)=\prod_{p\notin S(f)}I_{f}(s,p)(1-p^{-2}C_{p})^{-1}
\]
where the constant coefficient of the local factor $I_{f}(s,p)$ has been
normalized to be $1.$ Now let's set 
\begin{equation}
C_{p}=p-a_{p}.
\end{equation}
It is well known that 
\begin{equation}\label{absch}
|a_{p}|\leq 2g(f)p^{1/2}.
\end{equation}
for every $p\notin S(f)$ with $g(f)$ being the genus of the complex curve
defined by $f$.

\begin{proposition} Let $f(x,y)\in\BZ[x,y]$ 
be a non-singular polynomial then
the abscissa of convergence of $I_{f}(s)$ is $0.$
\end{proposition}

\begin{proof} Recall that
\[
\prod_{p}(1+r_{p})
\]
is absolutely convergent if and only $\sum_{p}\left| r_{p}\right| $
converges. Also $\sum_{p}\left| p^{-s}\right| $ converges if and only if $%
\Re (s)>1.$

We have
\[
 1+\lambda _{p}\frac{p^{-(s+1)}}{(1-p^{-(s+1)})} =
 1+
\frac{p^{-(s+1)}}{(1-p^{-(s+1)})}-\frac{pa_{p}}{(p^{2}-p+a_{p})}\frac{
p^{-(s+1)}}{(1-p^{-(s+1)})} . 
\]
Using the fact that $\left| a_{p}\right| \leq 2 p^{1/2}$ we deduce that the
abscissa of convergence is the same as
\[
\sum_{p}\left| \frac{p^{-(s+1)}}{(1-p^{-(s+1)})}\right| =\sum_{p}\left|
p^{-(s+1)}+\frac{p^{-2(s+1)}}{(1-p^{-(s+1)})}\right| .
\]
Hence $I_{f}(s)$ converges absolutely if and only if $\Re (s)>0.$
\end{proof}

A first meromorphic continuation in general:

\begin{proposition}\label{FirstCont} Let $f(x,y)\in\BZ[x,y]$ 
be a non-singular polynomial then
$I_{f}(s)$ has a meromorphic continuation to the region $\Re(s) >-1/2$.
\end{proposition}

\begin{proof} Firstly
\[
(1-p^{-(s+1)})\left( 1+\lambda _{p}\frac{p^{-(s+1)}}{(1-p^{-(s+1)})}\right)
=1-p^{-(s+1)}+\lambda _{p}p^{-(s+1)}
\]
and
\begin{eqnarray*}
-p^{-(s+1)}+\lambda _{p}p^{-(s+1)} &=&\frac{(p-1)(p-a_{p})p^{-(s+1)}-\left(
p^{2}-p+a_{p}\right) p^{-(s+1)}}{(p^{2}-C_{p})} \\
&=&\frac{-a_{p}p^{-s}}{(p^{2}-C_{p})}.
\end{eqnarray*}

Thus writing 
\[
I_{E}(s)=\zeta (s+1)\prod_{p}\left( 1-\frac{a_{p}p^{-s}}{(p^{2}-C_{p})}%
\right)
\]
continues $I_{f}(s)$ up to $\Re (s)>-1/2.$
\end{proof}

\bigskip
The results of this section up to this point are considered by Ono in his paper \cite{Ono}.

From now on we specialize to the case when our non-singular polynomial
$f(x,y)\in\BZ[x,y]$ defines an elliptic curve $E$, that is $g(f)=1$. We first
introduce some Dirichlet series related to the global Igusa zeta function
$I_f(s)$. First of all we recall the definition of the Hasse-Weil zeta function
of $E$.  
Let $\tilde E$ be the projectivized elliptic curve $E$ and
$N_{p^{m}}$ the number of points in $\tilde E(\BF_{p^{m}})$. Put
\[
\zeta (E,p,s)=\exp \left( \sum_{m=1}^{\infty }N_{p^{m}}\frac{p^{-ms}}{m}%
\right)
\]
Note
that this is the projective version of the curve in which case 
$N_{p}=C_{p}+1 $ 
since there is one extra point at infinity not picked up by $C_{p}.$ 
We then have the following explicit formula for this zeta function:
\[
\zeta (E,p,s)=\frac{1-a_{p}p^{-s}+p^{1-2s}}{(1-p^{-s})(1-p^{1-s})}.
\]

The Hasse-Weil zeta function is defined by 

\begin{equation}\label{HW}
\zeta (E,s)=\prod_{p\notin S(f)} (1-a_{p}p^{-s}+p^{1-2s})^{-1}.
\end{equation}

It is known that the Euler product of these zeta functions is meromorphic.

We recall some facts about $a_{p}:$

\begin{proposition}
\begin{itemize}
\item[(i)] $|a_p|\leq 2 p^{1/2}$

\item[(ii)] There exist algebraic integers 
$\pi_p $ and $\overline{\pi_p }$ in some
finite extension of $\BQ$ with $\pi_p \overline{\pi_p }=p$ 
and $a_{p}=\pi_p +\overline{\pi_p }$ 
and $\left| \pi_p \right| =\left| \overline{\pi_p }\right|
=p^{1/2}.$
\end{itemize}

\end{proposition}

Suppose now that $E$ has complex multiplication.
This means that its ring of endomorphisms ${\rm End}_\BC(E)$ is not equal to
$\BZ$ only. Consequently the tensor product 
$K:=\BQ\otimes_\BZ {\rm End}_\BC(E)$ is an imaginary quadratic number field.
Furthermore there is a choice of the $\pi_p$ and an idele class character
\begin{equation}
\chi_E: K^*\backslash \BA_k^*\to \BC^* 
\end {equation}
such that 
\begin{equation}
\zeta (E,p,s)={1-a_{p}p^{-s}+p^{1-2s}}=
\prod_{\wp | p} (1-\chi([\wp])N(\wp)^{-s})
\end{equation}

If $E$ does not have complex multiplication then 
by Wiles theorem there is a holomorphic modular form $g$ of weight $2$ such that $\zeta (E,s)=L(s,g,1)$.

We want to show next that we can continue $I_{E}(s)$ meromorphically to $\Re
(s)>-3/2.$ We shall use the following functions to do the continuation which
we have proved are meromorphic in section \ref{section: meromorphic
functions}:

\begin{proposition}
\label{Z_n is mero}Let 
\[
Z_{r}(s)=\prod_{p\notin S(f)}Z_{n,p}(s)
\]
where
\[
Z_{r,p}(s)=1-\left( \pi_p ^{r}+\overline{\pi_p }^{r}\right) p^{-s}+p^{r-2s}.
\]
Then $Z_{r}(s)$ is a meromorphic function on $\BC.$ 
\end{proposition}

In fact as we shall see it is enough
actually to have that they are meromorphic on $\Re (s)>1/2$.

Using this proposition we can now prove:

\begin{theorem}
\label{Meromorphic cont}The global Igusa zeta function of an elliptic curve $%
I_{E}(s)$ has meromorphic continuation to $\Re (s)>-3/2.$
\end{theorem}

\begin{proof}
By Proposition \ref{FirstCont} we need to show how to meromorphically continue:

\[
\prod_{p\notin S(f)}\left( 1-\frac{a_{p}p^{-s}}{(p^{2}-C_{p})}%
\right)
\]

We write for $p\notin S(f)$
\[
\left( 1-\frac{a_{p}p^{-s}}{(p^{2}-C_{p})}\right) =\left( 1-a_{p}p^{-2-s}-%
\frac{a_{p}p^{-2-s}\left( p^{-1}-a_{p}p^{-2}\right) }{\left(
1-p^{-1}+a_{p}p^{-2}\right) }\right) .
\]
Since
\[
\sum_{p}\frac{a_{p}p^{-2-s}\left( p^{-1}-a_{p}p^{-2}\right) }{\left(
1-p^{-1}+a_{p}p^{-2}\right) }
\]
converges absolutely on $\Re (s)>-3/2$ it will suffice to prove that

\[
W(s)=\prod_{p\notin S(f)}\left( 1-a_{p}p^{-s}\right)
\]
can be meromorphically continued to $\Re (s)>1/2.$ 

By Lemma \ref{deco1} we can write for $M\in \BN$:
\begin{eqnarray*}
\prod_{p\notin S(f)}\left( 1-a_{p}p^{-s}\right)
&=& \prod_{p\notin S(f)}\left( 1-(\pi_p +\overline{\pi_p })p^{-s}\right)\\
&=& \prod_{p\notin S(f)}\left( Q_M(\pi_p , \overline{\pi_p }, p, p^{-s}) + W_M(\pi_p , \overline{\pi_p }, p, p^{-s})\right)\\
&=& \left( \prod_{p\notin S(f)} Q_M(\pi_p , \overline{\pi_p }, p, p^{-s}) \right) \prod_{p\notin S(f)}\left( 1+\frac{W_M(\pi_p,\overline{\pi_p},p,p^{-s})}
{Q_M(\pi_p,\overline{\pi_p},p,p^{-s})} \right) \\
\end{eqnarray*}

Recall $Q_M$ was defined as a product of the polynomials 
\[
P^{[\varepsilon]}_{(r,n,m)}
=1-\varepsilon(u^r+v^r)X^nY^m+X^{r+2n}Y^{2m}
\]
Now
\[
P^{[\varepsilon]}_{(r,n,m)}(\pi_p,\overline{\pi_p },p,p^{-s})=
1-\left( \pi_p ^{r}+\overline{\pi_p }^{r}\right) p^{-ms+n}+p^{r-2s}=Z_{r,p}(ms-n)
\]

By Proposition \ref{Z_n is mero} this implies that 
\[
\left( \prod_{p\notin S(f)} Q_M(\pi_p , \overline{\pi_p }, p, p^{-s})\right)
\]
is a meromorphic function on $\BC$.

By Lemma \ref{eulel1} 
\[
\prod_{p\notin S(f)}\left( 1+\frac{W_M(\pi_p,\overline{\pi_p},p,p^{-s})}
{Q_M(\pi_p,\overline{\pi_p},p,p^{-s})}\right)
\]
converges absolutely for $\Re(s)>1/2+1/M$.

Therefore by taking $M$ to infinity we can meromorphically continue $W(s)$ to $\Re (s)>1/2.$ This in turn implies that the global Igusa zeta function of an elliptic curve $%
I_{E}(s)$ has meromorphic continuation to $\Re (s)>-3/2.$

\end{proof}

\bigskip

\section{Natural boundaries}

In this section we let $f(x,y)\in \BZ[x,y]$ be a nonsingular 
integer polynomial defining an elliptic curve
$E$. Recall that we did our meromorphic continuation using the function $$\prod_{p\notin S(f)} Q_M(\pi_p,\overline{\pi_p },p,p^{-s}).$$ This was built from a product and quotient of symmetric power L-functions which by \cite{CHT}, \cite{HS-BT} and \cite{Taylor} are meromorphic on $\BC$. Therefore the possible poles of our meromorphic function doing the continuation depends on the location of the poles and zeros of the symmetric power L-functions. Working under the assumption of the Generalized Riemann Hypothesis we can assume that zeros are located on $\Re(s)=1/2$. It will be sufficient to know that any poles of the symmetric power L-functions lie on the real axis.

Interest in meromorphically continuing the symmetric power $L$-functions arose because of Serre's observation that this would imply the Sato-Tate conjecture for the distribution of the values of $a_p$ in the case of $E$ not having complex multiplication. The distribution in the case of CM curves was known from class field theory.

In particular the results of  \cite{CHT}, \cite{HS-BT} and \cite{Taylor} imply the truth of the following:

\begin{theorem}
Suppose $E$ has no complex multiplication. Then $a_pp^{-1/2}/2$ is equidistributed in $[-1,1]$ with respect to the probability measure $\dfrac{2}{\pi}\sqrt{1-t^2}. $
\end{theorem}

In all cases we have the following:

\begin{theorem}{\label{dense}}
For each elliptic curve $E$ there is an infinite set ${\cal P}_E$ of primes $p$ such that $a_p/p^{-1/2}>1.$
\end{theorem}

\begin{theorem}
Assume that $I_{f}(s)$ has a meromorphic continuation to the whole of the
region defined by $\Re (s)>-3/2$ and 
also assume the generalized Riemann Hypothesis, then the line $\Re (s)=-3/2$ is a
natural boundary for $I_{E}(s)$ beyond which no further meromorphic
continuation is possible.
\end{theorem}

\proof We show that for infinitely many $p$ the equation 
\begin{equation}
 1-\frac{a_{p}p^{-s}}{(p^{2}-C_{p})} =0  \label{zeros}
\end{equation}
has a solutions at $s_{p}+\left( \theta _{p}+2\pi n\right) i/\log p$ 
where $s_{p}\in \BR$, $s_{p}>-3/2$. We also show that 
$s_{p}\rightarrow -3/2$ as 
$p\rightarrow\infty .$ 
We then need to establish that for sufficiently large $p$ these
zeros cannot be zeros or poles of the meromorphic functions 
being used to do the continuation.

Note first that for almost all primes $p$ we have $a_p\ne 0$ and 
(\ref{zeros}) is satisfied if and only if
\begin{equation}\label{zer1}
p^{s}=p^{-\frac{3}{2}}
\left(\frac{b_p}{1-p^{-1}+b_p p^{-\frac{3}{2}}}\right)
\end{equation}
where $b_p:=a_p p^{-\frac{1}{2}}$. Put
\begin{equation}\label{zer2}
r_p:=\frac{b_p}{1-p^{-1}+b_p p^{-\frac{3}{2}}}.
\end{equation}
By Theorem \ref{dense} there is an infinite set
${\cal P}_E$ of primes $p$ such that $b_p>1$ which in turn implies 
\begin{equation}\label{zer3}
b_p>\frac{1-p^{-1}}{1-p^{-\frac{3}{2}}}
\end{equation}
for every $p\in {\cal P}_E$. Inequality (\ref{zer3}) implies $r_p>1$,
hence we have
zeros of the form $s_{p}+\left( \theta _{p}+2\pi n\right) i/\log p$ where 
$s_{p}\in \BR$, $s_{p}>-3/2$. 
Since $r_p$ remains bounded as $p$ ranges over ${\cal P}_E$ we also
have $s_{p}\rightarrow -3/2$ as $p\rightarrow
\infty $ in ${\cal P}_\epsilon$.

We can therefore realise every point $-3/2+ai$ on the candidate natural
boundary as a limit point of these local zeros by letting $p\rightarrow
\infty $ and choosing $n_{p}$ with the property that $\left( \theta
_{p}+2\pi n_{p}\right) i/\log p\rightarrow a.$

Recall that the following expression was how we continued $I_{f}(s)$ to $\Re(s)>-3/2$ by considering $M$ tending to infinity:
$$I_{f}(s)=\prod_{p\notin S(f)} Q_M(\pi_p,\overline{\pi_p },p,p^{-s})\cdot \prod_{p\notin S(f)}
\left(1+\frac{W_M(\pi_p,\overline{\pi_p },p,p^{-s})}{Q_M(\pi_p,\overline{\pi_p },p,p^{-s})}\right). $$

We therefore need to confirm that the limit points of zeros we have identified cannot be killed by poles or zeros of 
 $
Z_{r}(ms+n)\ $or $\zeta (ms-n)$ where $n/m=-3/2.$ We established that in fact
the zeta function $Z_{n}(ms+n)$ can be rewritten in terms of the symmetric power $L$-functions so we only have to worry about zeros possibly cancelling our local zeros. Since we are assuming the generalized Riemann Hypothesis we can assume that a non-trivial zero $\rho $ of the Riemann zeta function or the symmetric power $L$-functions is
of the form $1/2+\gamma i.$ Suppose that $ms^{\prime }+n=\rho $ and $%
p^{-s^{\prime }}=-a_{p}^{-1}p^{2}+a_{p}^{-1}p-1.$ Note that $%
-a_{p}^{-1}p^{2}+a_{p}^{-1}p-1$ is a rational number, whilst $p^{-s^{\prime
}}=p^{-(1/2m-3/2+\gamma i/m)}$. So the only possible value of $m$ for which
this could happen is $m=1.$ The only global zeros that can possibly
interfere with local zeros are all on $\Re (s)=-1.$

This confirms then that $\Re (s)=-3/2$ is a natural boundary for $%
I_{E}(s).\square $


\end{document}